\documentclass{amsart}

\RequirePackage{fix-cm}
\usepackage{graphicx}
\usepackage{amsmath, amsopn,amstext,amscd,amsfonts,amssymb}
\usepackage{dsfont}
\usepackage{comment}
\usepackage[active]{srcltx}
\usepackage{graphicx, epsfig, subfig}

\usepackage{textcomp}

\usepackage{algorithm}

\makeatletter
\def\BState{\State\hskip-\ALG@thistlm}
\makeatother

\def\downbar#1{
\setbox10=\hbox{$#1$}
            \dimen10=\ht10 \advance\dimen10 by 2.5pt
            \ifdim \dimen10<15pt 
               \advance\dimen10 by -0.5pt
               \dimen11=\dimen10
               \advance\dimen10 by 2.5pt
               \lower \dimen11
            \else \lower \ht10 \fi
            \hbox {\hskip 1.5pt \vrule height \dimen10 depth \dp10}}
\def\upbar#1{
\setbox10=\hbox{$#1$}
            \dimen10=\ht10 \advance\dimen10 by \dp10 \advance\dimen10 by 2.5pt
            \ifdim \dimen10<15pt 
                \advance\dimen10 by 2pt \fi
            \raise 2.5pt \hbox {\hskip -1.5pt \vrule height \dimen10}}

\usepackage{multicol}
\usepackage{colortbl}
\usepackage{exscale}
\usepackage{amssymb,latexsym,amsthm,amsfonts,color,fancyhdr,mathrsfs}
\usepackage{lscape}
\usepackage{rotating}
\usepackage{caption}

\newtheorem{definition}{\bf Definition}[section]
\newtheorem{theorem}{\bf Theorem}[section]
\newtheorem{proposition}{\bf Proposition}[section]
\newtheorem{lemma}{\bf Lemma}[section]

\newtheorem{corollary}{\bf Corollary}[section]
\newtheorem{remark}{\bf Remark}[section]
\newtheorem{conjecture}{\bf Conjecture}[section]

\numberwithin{equation}{section}
\begin{document}
\title[On classical OPS and characterization theorems]{On classical orthogonal polynomials on lattices and some characterization theorems}

\author{K. Castillo}
\address{University of Coimbra, CMUC, Dep. Mathematics, 3001-501 Coimbra, Portugal}
\email{kenier@mat.uc.pt}
\author{D. Mbouna}
\address{University of Almer\'ia, Department of Mathematics, Almer\'ia, Spain}
\email{mbouna@ual.es}
\author{J. Petronilho}

\subjclass[2010]{42C05, 33C45}
\date{\today}
\keywords{Functional equation, regular functional, classical orthogonal polynomials, lattices, Racah polynomials, Askey-Wilson polynomials}

\maketitle

In this chapter are given necessary and sufficient conditions for the regularity of solutions of the functional equation appearing in the theory of classical orthogonal polynomials. In addition, we also present the functional Rodrigues formula and a closed formula for the recurrence coefficients. We finally used these results to solve some interesting research problems concerning characterization theorems.

\section{Introduction}\label{introduction}

Let $\mathcal{P}$ be the vector space of all polynomials with complex coefficients
and let $\mathcal{P}^*$ be its algebraic dual.
$\mathcal{P}_n$ denotes the space of all polynomials with degree less than or equal to $n$.
Define $\mathcal{P}_{-1}:=\{0\}$.
A simple set in $\mathcal{P}$ is a sequence $(P_n)_{n\geq0}$ such that
$P_n\in\mathcal{P}_n\setminus\mathcal{P}_{n-1}$ for each $n$.
A simple set $(P_n)_{n\geq0}$ is called an orthogonal polynomial sequence (OPS)
with respect to ${\bf u}\in\mathcal{P}^*$ if 
$$
\langle{\bf u},P_nP_k\rangle=h_n\delta_{n,k}\quad(n,k=0,1,\ldots;\;h_n\in\mathbb{C}\setminus\{0\})\;,
$$
where $\langle{\bf u},f\rangle$ is the action of ${\bf u}$ on $f\in\mathcal{P}$. 
${\bf u}$ is called regular, or quasi-definite, if
there exists an OPS with respect to it. The left multiplication of a functional ${\bf u}$ by a polynomial $\phi$ is defined by
$$
\left\langle \phi {\bf u}, f  \right\rangle =\left\langle {\bf u},\phi f  \right\rangle \quad (f\in \mathcal{P}).
$$
A sequence $({\bf a}_n)_{n\geq 0}$ is called dual basis for the sequence of simple set polynomials $(P_n)_{n\geq 0}$ if $\left\langle {\bf a}_n,P_m\right\rangle=\delta_{n,m}$, for each $n,m=0,1,\ldots\;.$
Consequently, if $(P_n)_{n\geq0}$ is a monic OPS with respect to ${\bf u}\in\mathcal{P}^*$, then the corresponding dual basis is explicitly given by 
\begin{align}\label{expression-an}
{\bf a}_n =\left\langle {\bf u} , P_n ^2 \right\rangle ^{-1} P_n{\bf u}.
\end{align}
Any functional ${\bf u} \in \mathcal{P}^*$ ($\mathcal{P}$ being endowed with an appropriate strict inductive limit topology --- see e.g. \cite{KCJPLN, M1991}) can be written  as 
\begin{align*}
{\bf u} = \sum_{n=0} ^{\infty} \left\langle {\bf u}, P_n \right\rangle {\bf a}_n,
\end{align*}
in the sense of the weak dual topology in $\mathcal{P}^*$, where $({\bf a}_n)_{n\geq 0}$ is the dual basis for the sequence of polynomials $(P_n)_{n\geq 0}$. These are some basic facts of the algebraic theory on orthogonal polynomials due to P. Maroni (see \cite{M1985, M1988, M1991, M1994}). For a recent survey on the subject we refer the reader to \cite{KCJPLN}.

\subsection{Lattices}
\begin{definition}
A lattice in this framework is a mapping $x(s)$ given by
\begin{equation}
\label{xs-def}
x(s):=\left\{
\begin{array}{ccl}
\mathfrak{c}_1 q^{-s} +\mathfrak{c}_2 q^s +\mathfrak{c}_3 & {\rm if} &  q\neq1 \;,\\ [0.75em]
\mathfrak{c}_4 s^2 + \mathfrak{c}_5 s +\mathfrak{c}_6 & {\rm if} &  q =1\;,
\end{array}
\right.
\end{equation}
$s\in\mathbb{C}$, where $q>0$ (fixed) and $\mathfrak{c}_j$ ($1\leq j\leq6$) are (complex) constants,
that may depend on $q$, such that $(\mathfrak{c}_1,\mathfrak{c}_2)\neq(0,0)$ if $q\neq1$,
and $(\mathfrak{c}_4,\mathfrak{c}_5,\mathfrak{c}_6)\neq(0,0,0)$ if $q=1$.
\end{definition}
In the case $q=1$, the lattice is called quadratic if $\mathfrak{c}_4\neq0$,
and linear if $\mathfrak{c}_4=0$;
and in the case $q\neq1$, it is called $q-$quadratic if $\mathfrak{c}_1\mathfrak{c}_2\neq0$,
and $q-$linear if $\mathfrak{c}_1\mathfrak{c}_2=0$ (cf. \cite{ARS1995}).
Notice that
$$
\frac{x\big(s+\frac12\big)+x\big(s-\frac12\big)}{2}=\alpha x(s)+\beta\;,
$$
where $\alpha$ and $\beta$ are given by
\begin{equation}\label{alpha-beta}
\alpha:=\frac{q^{1/2}+q^{-1/2}}{2}\;,\quad
\beta:=\left\{
\begin{array}{ccl}
(1-\alpha)\mathfrak{c}_3 & {\rm if} &  q\neq1 \;,\\ [0.75em]
\mathfrak{c}_4/4 & {\rm if} &  q =1\;.
\end{array}
\right.
\end{equation}
The lattice $x(s)$ fulfills (cf. \cite{ARS1995}): 
\begin{align*}
\frac{x(s+n)+x(s)}{2}&=\alpha_n x_n (s) +\beta_n \,, \\
x(s+n)-x(s)&=\gamma_n \nabla x_{n+1} (s)\;,
\end{align*}
for each $n=0,1,\ldots$, where $x_{\mu}(s):=x\big(s+\mbox{$\frac\mu2$}\big)$, $\nabla f(s):=f(s)-f(s-1)$,
and $(\alpha_n)_{n\geq0}$, $(\beta_n)_{n\geq0}$, and $(\gamma_n)_{n\geq0}$ are sequences of numbers
generated by the following system of difference equations
\begin{align}
&\alpha_0 =1\;,\quad \alpha_1=\alpha\;,\quad\alpha_{n+1} -2\alpha\alpha_n +\alpha_{n-1} =0 \label{1.2}\\
&\beta_0 =0\;,\quad \beta_1 =\beta\;,\quad\beta_{n+1} -2\beta_n + \beta_{n-1} =2\beta\alpha_{n} \label{1.2b} \\ 
&\gamma_0 =0\;,\quad \gamma_1 =1\;,\quad \gamma_{n+1} -\gamma_{n-1} =2\alpha_n\;, \label{1.2c}
\end{align}
for each $n=1,2,\ldots$. The explicit solutions of these difference equations are
\begin{align}
& \alpha_n = \frac{q^{n/2} +q^{-n/2}}{2}\,, \label{alpha-n} \\
& \beta_n = \displaystyle\left\{
\begin{array}{ccl}
\displaystyle\beta\,\left(\frac{q^{n/4}-q^{-n/4}}{q^{1/4}-q^{-1/4}}\right)^2 & \mbox{\rm if} & q\neq1 \\ [1em]
\beta\,n^2 & \mbox{\rm if} & q=1\,,
\end{array}
\right. \label{beta-n} \\
& \gamma_n = \displaystyle\left\{
\begin{array}{ccl}
\displaystyle\frac{q^{n/2}-q^{-n/2}}{q^{1/2}-q^{-1/2}} & \mbox{\rm if} & q\neq1 \\ [1em]
n & \mbox{\rm if} & q=1\;.
\end{array}
\right. \label{gamma-n}
\end{align}
We assume that $\alpha_{-1}:=\alpha$ and $\gamma_{-1}:=-1$, consistently with \eqref{alpha-n} and \eqref{gamma-n}.

\subsection{Operators on lattices}

\begin{definition}\label{def-DxSxNUL} 
Let $x(s)$ be a lattice given by \eqref{xs-def}.  
The $x-$derivative operator on $\mathcal{P}$, $\mathrm{D}_x$, 
and the $x-$average operator on $\mathcal{P}$, $\mathrm{S}_x$, 
are the operators on $\mathcal{P}$ defined for each $f\in\mathcal{P}$ so that $\deg(\mathrm{D}_x f)=\deg f-1$, $\deg(\mathrm{S}_x f)=\deg f$, and
\begin{align}
\mathrm{D}_x f(x(s))&= \frac{f\big(x(s+\frac{1}{2})\big)-f\big(x(s-\frac{1}{2})\big)}{x(s+\frac{1}{2})-x(s-\frac{1}{2})}\;, \label{AWxsG}  \\
\mathrm{S}_x f(x(s))&= \frac{f\left(x\big(s+\frac{1}{2}\big)\right) + f\left(x\big(s-\frac{1}{2}\big)\right)}{2}\,,
\label{AverOperG}
\end{align}
where it is understood that $\mathrm{D}_x f=f'$ and  $\mathrm{S}_x f=f$ whenever $x(s)=\mathfrak{c}_6$.
\end{definition}

The relations \eqref{AWxsG} and \eqref{AverOperG} appear in \cite[(3.2.4)-(3.2.5)]{NSU1991} up to a shift in the variable $s$. 
The operators  $\mathrm{D}_x$ and $\mathrm{S}_x$ on $\mathcal{P}$ induce two operators on the dual space $\mathcal{P}^*$, namely
$\mathbf{D}_x:\mathcal{P}^*\to\mathcal{P}^*$ and $\mathbf{S}_x:\mathcal{P}^*\to\mathcal{P}^*$, via the following definition (cf. \cite{FK-NM2011}):

\begin{definition} 
Let $x(s)$ be a lattice given by \eqref{xs-def}. 
For each ${\bf u}\in\mathcal{P}^*$, the functionals $\mathbf{D}_x{\bf u}\in\mathcal{P}^*$
and $\mathbf{S}_x{\bf u}\in\mathcal{P}^*$ are defined by
\begin{equation}\label{PNUL-def-Dxu-Sxu}
\langle \mathbf{D}_x{\bf u},f\rangle:=-\langle {\bf u},\mathrm{D}_x f\rangle\; ,\quad
\langle \mathbf{S}_x{\bf u},f\rangle:=\langle {\bf u},\mathrm{S}_x f\rangle\quad (f\in\mathcal{P})\,.
\end{equation}
We call $\mathbf{D}_x{\bf u}$ the $x-$derivative of ${\bf u}$ and $\mathbf{S}_x{\bf u}$ the $x-$average of ${\bf u}$.
\end{definition}
Hereafter, $z:=x(s)$ being a lattice given by (\ref{xs-def}). We consider two fundamental polynomials, 
$\texttt{U}_1$ and $\texttt{U}_2$, introduced in \cite{FK-NM2011}, defined by

\begin{equation} \label{U1-simples-bis}
\texttt{U}_1 (z) =\left\{
\begin{array}{lcl}
(\alpha^2-1)\big(z-\mathfrak{c}_3\big) & \mbox{\rm if} & q\neq1 \;, \\ [0.5em]
\frac{1 }{2}\mathfrak{c}_4 & \mbox{\rm if} & q=1 \,,
\end{array}
\right.
\end{equation}
and
\begin{equation} \label{U2-simples-bis}
\texttt{U}_2(z) =\left\{
\begin{array}{lcl}
(\alpha^2-1)\big((z-\mathfrak{c}_3)^2-4\mathfrak{c}_1\mathfrak{c}_2\big) & \mbox{\rm if} & q\neq1 \;, \\ [0.5em]
\mathfrak{c}_4(z-\mathfrak{c}_6)+\frac14\mathfrak{c}_5^2& \mbox{\rm if} & q=1 \;.
\end{array}
\right.
\end{equation}

We point out some useful properties (see \cite{KCDMJP2022classical, MFN-S2017, N-SKF2015} and references therein). 

\begin{proposition}\label{propertyDxSx}
Let $f,g\in\mathcal{P}$ and ${\bf u}\in\mathcal{P}^*$. Then the following properties hold:\rm
\begin{align}
 &\mathrm{D}_x \big(fg\big)= \big(\mathrm{D}_x f\big)\big(\mathrm{S}_x g\big)+\big(\mathrm{S}_x f\big)\big(\mathrm{D}_x g\big)\;, \label{1.3} \\
 &\mathrm{S}_x \big( fg\big)= \big(\mathrm{D}_x f\big) \big(\mathrm{D}_x g\big)\texttt{U}_2 +\big(\mathrm{S}_x f\big) \big(\mathrm{S}_x g\big)\;, \label{1.4} \\
&f\mathrm{S}_xg
=\mathrm{S}_x\Big(\big(\mathrm{S}_xf-\alpha^{-1}\texttt{U}_1\,\mathrm{D}_xf\big)g\Big)
-\alpha^{-1}\texttt{U}_2\mathrm{D}_x\big(g\mathrm{D}_xf\big)\;, \label{PropNUL2} \\
 &f\mathrm{D}_x g =\mathrm{D}_x\Big(\big(\mathrm{S}_xf-\alpha^{-1}\texttt{U}_1\,\mathrm{D}_xf\big)g\Big)
-\alpha^{-1}\mathrm{S}_x\big(g\mathrm{D}_xf\big)\;, \label{PropNUL1} \\
&\mathbf{D}_x(f{\bf u})
=\left(\mathrm{S}_x f -\alpha^{-1}\texttt{U}_1\mathrm{D}_x f\right)\mathbf{D}_x {\bf u} +
\alpha^{-1} \big(\mathrm{D}_x f\big) \mathbf{S}_x {\bf u}\;, \label{a} \\
&\mathbf{S}_x (f {\bf u}) 
= \big(\alpha \texttt{U}_2 -\alpha^{-1}\texttt{U}_1^2  \big)(\mathrm{D}_x f)\mathbf{D}_x {\bf u} +
\big(\mathrm{S}_x f +\alpha^{-1}\texttt{U}_1\mathrm{D}_x f  \big)\mathbf{S}_x {\bf u}\;,  \label{1} \\
&\mathrm{D}_x ^n \mathrm{S}_x f = \alpha_n \mathrm{S}_x \mathrm{D}_x ^n f +\gamma_n \mbox{\rm $\texttt{U}_1$} \mathrm{D}_x ^{n+1}f \quad (n=0,1,\ldots)\;,\\
&\alpha \mathbf{D}_x ^n \mathbf{S}_x {\bf u}
= \alpha_{n+1} \mathbf{S}_x \mathbf{D}_x^n {\bf u}
+\gamma_n\mbox{\rm $\texttt{U}_1$}\mathbf{D}_x^{n+1}{\bf u}\quad (n=0,1,2,\ldots)\label{def Dx^nSx-u}\;.
\end{align}
\end{proposition}

The following results can be proved by mathematical induction on $n$ (see alternatively \cite{KCDMJP2022classical}).
\begin{proposition}
For the lattice $\;x(s)=\mathfrak{c}_1q^{-s}+\mathfrak{c}_2q^s+\mathfrak{c}_3$, the following holds:
\begin{align}
\mathrm{D}_x z^n &=\gamma_n z^{n-1}+u_nz^{n-2}+v_nz^{n-3}+\cdots\;, \label{Dx-xn}\\
\mathrm{S}_x z^n &=\alpha_n z^n+\widehat{u}_nz^{n-1}+\widehat{v}_nz^{n-2}+\cdots\;, \label{Sx-xn}
\end{align}
for each $n=0,1,\ldots$, where $\alpha_n$ and $\gamma_n$ are given by \eqref{alpha-n} and \eqref{gamma-n}, and
\begin{align}
u_n &:=\big(n\gamma_{n-1}-(n-1)\gamma_n\big)\mathfrak{c}_3\;, \label{unD} \\
v_n &:=\big(n\gamma_{n-2}-(n-2)\gamma_n\big)\mathfrak{c}_1\mathfrak{c}_2 \label{vnD} \\
&\qquad+\mbox{$\frac12$}\big(n(n-1)\gamma_{n-2}-2n(n-2)\gamma_{n-1}+(n-1)(n-2)\gamma_{n}\big)\mathfrak{c}_3^2 \;, \nonumber\\
\widehat{u}_n &:=n(\alpha_{n-1}-\alpha_n)\mathfrak{c}_3\;, \label{unS} \\
\widehat{v}_n &:=n(\alpha_{n-2}-\alpha_n)\mathfrak{c}_1\mathfrak{c}_2
+n(n-1)(\alpha-1)\alpha_{n-1}\mathfrak{c}_3^2\;. \label{vnS}
\end{align}
\end{proposition}

We introduce an operator $\mathrm{T}_{n,k}:\mathcal{P}\to\mathcal{P}$
($n=0,1,\ldots;\, k=0,1,\ldots, n$),
defined for each $f\in\mathcal{P}$ as follows: if $n=k=0$, set
\begin{align}\label{T00}
\mathrm{T}_{0,0}f &:=f\;;
\end{align}
and if $n\geq1$ and $0\leq k\leq n$, define recurrently
\begin{align}\label{Tnk}
\mathrm{T}_{n,k}f&:= \mathrm{S}_x \mathrm{T}_{n-1,k}f
-\frac{\gamma_{n-k}}{ \alpha_{n-k}}\texttt{U}_1 \mathrm{D}_x \mathrm{T}_{n-1,k}f
+\frac{1}{\alpha_{n+1-k}} \mathrm{D}_x \mathrm{T}_{n-1,k-1}f\,,
\end{align}
with the conventions $\mathrm{T}_{n,k}f:=0$ whenever $k>n$ or $k<0$. 
Note that
$$
\deg \mathrm{T}_{n,k}f\leq\deg f-k\;.
$$
We are ready to state the following

\begin{theorem}[Leibniz's formula]\label{Leibniz-rule-NUL}
Let ${\bf u}\in\mathcal{P}^*$ and $f\in\mathcal{P}$. Then
\begin{align}
\mathbf{D}_x ^n \big(f{\bf u}\big)
=\sum_{k=0}^{n} \mathrm{T}_{n,k}f\, \mathbf{D}_x^{n-k} \mathbf{S}_x^k {\bf u}
 \quad (n=0,1,\ldots),\label{leibnizfor-NUL}
\end{align}
where $\mathrm{T}_{n,k}f$ is a polynomial defined by \eqref{T00}--\eqref{Tnk}.
\end{theorem}

\begin{corollary}\label{LeibnizCor1}
Consider the lattice $x(s):=\mathfrak{c}_1 q^{-s} +\mathfrak{c}_2 q^s +\mathfrak{c}_3$. 
Let ${\bf u}\in\mathcal{P}^*$ and $f\in\mathcal{P}_2$. Write $f(z)=az^2+bz+c\,,$ with $a,b,c\in\mathbb{C}$. Then
\begin{align}\label{leibnizfor-degree-pi-2}
\mathbf{D}_x ^n (f{\bf u})
&=\left(\frac{a\alpha}{\alpha_n\alpha_{n-1}}\,(z-\mathfrak{c}_3)^2
+\frac{f'(\mathfrak{c}_3)}{\alpha_n}(z-\mathfrak{c}_3) +f(\mathfrak{c}_3)+\frac{4a(1-\alpha^2)\gamma_n \mathfrak{c}_1\mathfrak{c}_2}{\alpha_{n-1}}\right) \mathbf{D}_x^n{\bf u}  \\
&\quad+\frac{\gamma_n}{\alpha_n}\left( \frac{a(\alpha_n +\alpha \alpha_{n-1}) }{\alpha_{n-1}^2}\,(z-\mathfrak{c}_3)
+f'(c_3)\right) \mathbf{D}_x^{n-1}\mathbf{S}_x{\bf u}\nonumber \\
&\quad+\frac{a\gamma_n\gamma_{n-1}}{\alpha_{n-1}^2}\,\mathbf{D}_x^{n-2}\mathbf{S}_x^2{\bf u}  \nonumber\;.
\end{align}

\end{corollary}

\section{Classical OPS on lattices}\label{SecRodrigues}

This section concerns classical OPS on lattices and their associated regular functionals. 
We start by reviewing some basic definitions and needed properties. 

\begin{definition}\label{NUL-def}
Let $x(s)$ be a lattice given by \eqref{xs-def}. 
${\bf u}\in\mathcal{P}^*$ is called $x-$classical if it is regular and there exist nonzero polynomials
$\phi\in\mathcal{P}_2$ and $\psi\in\mathcal{P}_1$
such that
\begin{equation}\label{NUL-Pearson}
\mathbf{D}_x(\phi{\bf u})=\mathbf{S}_x(\psi{\bf u})\;.
\end{equation}
An OPS with respect to a $x-$classical functional will be called a $x-$classical OPS
(or a classical OPS on the lattice $x$).
\end{definition}
Definition \ref{NUL-def} appears in \cite{FK-NM2011}. We will refer to (\ref{NUL-Pearson}) as {\it $x-$Geronimus--Pearson functional equation on the lattice $x$,} or, simply, {\it $x-$GP functional equation}. Our principal goal is to state necessary and sufficient conditions, involving only $\phi$ and $\psi$, such that a given functional ${\bf u}\in\mathcal{P}^*$ satisfying the
$x-$GP functional equation (\ref{NUL-Pearson}) becomes regular.
In order to move on, we need to introduce notations and to state some preliminary needed properties.

We denote by $P_n ^{[k]}$ the monic polynomial of degree $n$ defined by
\begin{align}
P_n ^{[k]} (z):=\frac{D_x ^k P_{n+k} (z)}{ \prod_{j=1} ^k \gamma_{n+j}} =\frac{\gamma_{n} !}{\gamma_{n+k} !} D_x ^k P_{n+k} (z) \quad (k,n=0,1,\ldots)\;. \label{Pnkx}
\end{align}
Here, as usual, it is understood that $\mathrm{D}_x ^0 f=f $, empty product equals one, and 
$$\gamma_0 !:=1\;,\quad \gamma_{n+1} !:=\gamma_1 ...\gamma_n \gamma_{n+1} \quad(n=0,1,\ldots)\;.$$
\begin{definition}
Let $\phi\in\mathcal{P}_2$ and $\psi\in\mathcal{P}_1$.
$(\phi,\psi)$ is called an $x-$admissible pair if
$$
d_n\equiv d_n(\phi,\psi,x):=\mbox{$\frac12$}\,\gamma_n\,\phi^{\prime\prime}+\alpha_n\psi'\neq0 \quad (n=0,1,\ldots)\;.
$$
\end{definition}
This is an analogous for lattices of the corresponding definitions for the $\mathrm{D}-$classical and $(q,\omega)-$classical  cases (cf. \cite{M1991,MP1994,RKDP2020}).

Following \cite{FK-NM2011}, given ${\bf u}\in\mathcal{P}^*$, $\phi\in\mathcal{P}_2$, and $\psi\in\mathcal{P}_1$,
we define recursively polynomials $\phi^{[k]}\in\mathcal{P}_2$ and $\psi^{[k]}\in\mathcal{P}_1$ by
\begin{align}
& \phi^{[0]}:=\phi\;,\quad \psi^{[0]}:=\psi\;, \label{def-phi0psi0}\\
& \phi^{[k+1]}:=\mathrm{S}_x\phi^{[k]}+\texttt{U}_1\mathrm{S}_x\psi^{[k]}+\alpha\texttt{U}_2\mathrm{D}_x\psi^{[k]}\;, \label{def-phik}\\
& \psi^{[k+1]}:=\mathrm{D}_x\phi^{[k]}+\alpha\mathrm{S}_x\psi^{[k]}+\texttt{U}_1\mathrm{D}_x\psi^{[k]}\;, \label{def-psik}
\end{align}
and functionals ${\bf u}^{[k]}\in\mathcal{P}^*$ by
\begin{equation}\label{uk-func-Dx}
{\bf u}^{[0]}:={\bf u}\;,\quad
{\bf u}^{[k+1]}:=\mathbf{D}_x\big(\texttt{U}_2\psi^{[k]}{\bf u}^{[k]}\big)-\mathbf{S}_x\big(\phi^{[k]}{\bf u}^{[k]}\big)
\end{equation}
($k=0,1,\ldots$).
${\bf u}^{[k]}$ may be seen as the higher order $x-$derivative of ${\bf u}$.
Next, we provide explicit representations for the polynomials $\phi^{[k]}$ and $\psi^{[k]}$.

\begin{proposition}
Consider the lattice $x(s):=\mathfrak{c}_1 q^{-s} +\mathfrak{c}_2 q^s +\mathfrak{c}_3$.
Let $\phi\in\mathcal{P}_2$ and $\psi\in\mathcal{P}_1$, so there are $a,b,c,d,e\in\mathbb{C}$ such that
$$\phi(z)=az^2+bz+c\;,\quad \psi(z)=dz+e\,.$$ 
Then the polynomials $\phi^{[k]}$ and $\psi^{[k]}$
defined by \eqref{def-phi0psi0}--\eqref{def-psik} are given by 
\begin{align}
\psi^{[k]}(z)&=\big( a\gamma_{2k}+d\alpha_{2k} \big)(z-\mathfrak{c}_3)+\phi'(\mathfrak{c}_3)\gamma_k+\psi(\mathfrak{c}_3)\alpha_k\;, \label{psi-explicit}\\
\phi^{[k]}(z)&=\big(d(\alpha^2-1)\gamma_{2k}+a\alpha_{2k}\big)
\big((z-\mathfrak{c}_3)^2-2\mathfrak{c}_1\mathfrak{c}_2\big)  \label{phi-explicit} \\
&\quad +\big(\phi'(\mathfrak{c}_3)\alpha_k+\psi(\mathfrak{c}_3)(\alpha^2-1)\gamma_k\big)(z-\mathfrak{c}_3)
+ \phi(\mathfrak{c}_3)+2a\mathfrak{c}_1\mathfrak{c}_2, \nonumber
\end{align}
for each $k=0,1,2\ldots$.
\end{proposition}

\begin{proof}
By mathematical induction on $n$.
\end{proof}

\begin{theorem}\label{x-admissible}\cite{KCDMJP2022classical}
Let ${\bf u}\in\mathcal{P}^*$. Suppose that ${\bf u}$ is regular and satisfies $(\ref{NUL-Pearson})$, 
where $(\phi,\psi)\in\mathcal{P}_2\times\mathcal{P}_1 \setminus \{(0,0)\}$.
Then $(\phi,\psi)$ is a $x-$admissible pair and ${\bf u}^{[k]}$ is regular for each $k=1,2,\ldots$.
Moreover, if $(P_n)_{n\geq0}$ is the monic OPS with respect to ${\bf u}$, then
$\big(P_n^{[k]}\big)_{n\geq0}$ is the monic OPS with respect to ${\bf u}^{[k]}$.
\end{theorem}
The converse of this theorem is known in the following form (see \cite{KCDM2022c}).
\begin{theorem}
If $(P_n)_{n\geq 0}$ and $(P_n ^{[k]})_{n\geq 0}$, for some $k$, $k=1,2,\ldots$ are both orthogonal polynomials sequences, then $(P_n)_{n\geq 0}$ is $x$-classical OPS.
\end{theorem}

We can now state the main result of this paragraph (see \cite{KCDMJP2022classical} for its proof).
\begin{theorem}\label{main-Thm1}
Consider the lattice
$$
x(s)=\mathfrak{c}_1 q^{-s} +\mathfrak{c}_2 q^s +\mathfrak{c}_3\,. 
$$
Let ${\bf u}\in\mathcal{P}^*\setminus\{{\bf 0}\}$ and suppose that there exist $(\phi,\psi)\in\mathcal{P}_2\times\mathcal{P}_1\setminus\{(0,0)\}$ 
such that 
\begin{equation}\label{NUL-PearsonMainThm1}
\mathbf{D}_{x}(\phi{\bf u})=\mathbf{S}_{x}(\psi{\bf u})\;.
\end{equation}
Set $\phi(z)=az^2+bz+c$, $\psi(z)=dz+e$ ($a,b,c,d,e \in \mathbb{C}$), $d_{n}=a\gamma_{n}+d\alpha_{n}$, $e_n=\phi'(\mathfrak{c}_3)\gamma_n+\psi(\mathfrak{c}_3)\alpha_n$, and $\phi^{[n]}$ and $\psi^{[n]}$ be given by \eqref{psi-explicit}--\eqref{phi-explicit}. 
Then,  ${\bf u}$ is regular if and only if $(\phi,\psi)$ is a $x-$admissible pair and $\psi^{[n]}\nmid\phi^{[n]}$ for each $n=0,1,\ldots$; 
that is to say, ${\bf u}$ is regular  if and only if the following conditions hold:
\begin{equation}\label{le1a}
d_n\neq0\;,\quad \phi^{[n]}\left(\mathfrak{c}_3 -\frac{e_n}{d_{2n}}\right)\neq0\quad(n=0,1,\ldots)\,.
\end{equation}
Under such conditions, the monic OPS $(P_n)_{n\geq 0}$ with respect to ${\bf u}$ satisfies 
\begin{equation}\label{ttrr-Dx}
P_{n+1}(z)=(z-B_n)P_n(z)-C_nP_{n-1}(z) \quad(n=0,1,\ldots),
\end{equation}
with $P_{-1}(z)=0$, where the recurrence coefficients are given by
\begin{align}
B_n & =\mathfrak{c}_3+ \frac{\gamma_n e_{n-1}}{d_{2n-2}}
-\frac{\gamma_{n+1}e_n}{d_{2n}}\, ,\label{Bn-Dx} \\
C_{n+1} & =-\frac{\gamma_{n+1}d_{n-1}}{d_{2n-1}d_{2n+1}}\phi^{[n]}\left(\mathfrak{c}_3 -\frac{e_{n}}{d_{2n}}\right)\label{Cn-Dx}
\end{align}
$(n=0,1,\ldots)$.
Moreover, the following functional Rodrigues formula holds:
\begin{align}\label{RodThemMain}
P_n {\bf u} = k_n\mathbf{D}_x^n {\bf u}^{[n]}\;, \quad
k_n := (-\alpha)^{-n} \prod_{j=1} ^n d_{n+j-2} ^{-1} 
\quad(n=0,1,\ldots)\,.
\end{align}
\end{theorem}

The following asymptotic behaviour of the above sequence $(B_n)_{n\geq 0}$ is useful (see \cite{DMthesis}).
\begin{corollary}\label{assymptotic-behaviour-coef-ttrr-Dx}

Assume the hypothesis (and hence the conclusions) of Theorem \ref{main-Thm1}. Then
\begin{align}\label{SBn-converges}
S_n := \sum_{j=0} ^{n-1} (B_j -\mathfrak{c}_3) = -\frac{\gamma_n e_{n-1}}{d_{2n-2}}, \quad n=0,1,2,\ldots.
\end{align}
In addition, setting $u:=(q^{1/2} -q^{-1/2})^{-1}$, the following holds:
\begin{itemize}
\item[a)] For $0<q<1$ and $d-2au \neq 0$, we have 
\begin{align}
&\lim_{n \rightarrow \infty} (B_n -\mathfrak{c}_3) = \lim_{n \rightarrow \infty} q^{\pm n/2}(B_n -\mathfrak{c}_3)=0\;, \label{q-is-btwn01a}\\
&\lim_{n \rightarrow \infty} q^{-n} (B_n -\mathfrak{c}_3 )=-\frac{q^{-1/2}\Big(\psi(\mathfrak{c}_3) -4\alpha u^2\phi'(\mathfrak{c}_3)\Big)}{u(d-2au)}\;, \label{q-is-btwn01b}\\
&S:= \sum_{j=0} ^{\infty} (B_j -\mathfrak{c}_3) =\frac{\psi(\mathfrak{c}_3 )-2u \phi'(\mathfrak{c}_3)}{(q-1)(d-2au)}\;. \label{q-is-btwn01c}
\end{align} 
\item[b)]For $1<q<\infty$ and $d+2au \neq 0$, we have 
\begin{align}
&\lim_{n \rightarrow \infty} (B_n -\mathfrak{c}_3) = \lim_{n \rightarrow \infty} q^{\pm n/2}(B_n -\mathfrak{c}_3)=0\; ,\label{q-is-big1a}\\
&\lim_{n \rightarrow \infty} q^n (B_n -\mathfrak{c}_3 )=\frac{q^{1/2}\Big(\psi(\mathfrak{c}_3) -4\alpha u^2\phi'(\mathfrak{c}_3)\Big)}{u(d+2au)} \label{q-is-big1b}\\ 
&S=\sum_{j=0} ^{\infty} (B_j -\mathfrak{c}_3)=\frac{\psi(\mathfrak{c}_3 )+2 u\phi'(\mathfrak{c}_3)}{(q^{-1}-1)(d+2au)}\;.\label{q-is-big1c}
\end{align}
\end{itemize}
\end{corollary}

Applying a limiting process (as $q\to1$) on Theorem \ref{main-Thm1}, we obtain the following result for quadratic lattices.

\begin{theorem}\label{mainthmquadratics}
Consider the lattice  $$x(s)=4\beta s^2 +\mathfrak{c}_5 s +\mathfrak{c}_6\,.$$ 
Let ${\bf u}\in\mathcal{P}^*\setminus\{{\bf 0}\}$ and suppose that there exist $(\phi,\psi)\in\mathcal{P}_2\times\mathcal{P}_1\setminus\{(0,0)\}$ such that 
\begin{equation}\label{NUL-PearsonMainThm2}
\mathbf{D}_{x}(\phi{\bf u})=\mathbf{S}_{x}(\psi{\bf u})\;.
\end{equation}
Set $\phi(z):=az^2+bz+c$ and $\psi(z):=dz+e$. 
Then {\bf u} is regular if and only if
\begin{equation}\label{le1a-quadratic}
d_n\neq0\;,\quad \phi^{[n]}\left(-\beta n^2 -\frac{e_n}{d_{2n}}\right)\neq0 \quad (n=0,1,\ldots)\,, 
\end{equation}
where $d_n :=an+d$, $e_n:=bn+e+2\beta dn^2$, and 
$$\phi ^{[n]}(z)=az^2 +(b+6\beta nd_n)z+ \phi(\beta n^2)+2\beta n\psi(\beta n^2)-\frac{n}{4}\left( 16\beta \mathfrak{c}_6 -\mathfrak{c}_5 ^2\right)d_n\,.$$ 
Under these conditions, the monic OPS $(P_n)_{n\geq 0}$ with respect to ${\bf u}$ satisfies \eqref{ttrr-Dx} with 
\begin{align}
B_n &= \frac{ne_{n-1}}{d_{2n-2}} -\frac{(n+1)e_n}{d_{2n}} -2\beta n(n-1),   \label{Bn-quadratic}\\
C_{n+1} &=-\frac{(n+1)d_{n-1}}{d_{2n-1}d_{2n+1}}\phi ^{[n]}\left(-\beta n^2 -\frac{e_n}{d_{2n}}  \right)\label{Cn-quadratic}
\end{align} 
$(n=0,1,\ldots)$.
Moreover, the following functional Rodrigues formula holds:
\begin{align}\label{RodThemMain2}
P_n {\bf u} = k_n\mathbf{D}_x^n {\bf u}^{[n]}\;, \quad
k_n := (-1)^{n} \prod_{j=1} ^n d_{n+j-2} ^{-1} 
\quad(n=0,1,\ldots)\,.
\end{align}
\end{theorem}
The following corollary can be easily proved.
\begin{corollary}
Under the assumptions of Theorem \ref{mainthmquadratics} (and hence the conclusions), the following hold.
\begin{align}\label{Bn-limit-quadratic}
\lim_{n\rightarrow \infty}~ \frac{C_{n+1}}{n^4} =\left\{
\begin{array}{lcl}
\beta^2, &  a\neq 0\;,\\ [7pt]
16\beta^2, &  a=0\;.
\end{array}
\right. \;\quad
\lim_{n\rightarrow \infty}~ \frac{B_n}{n^2} =\left\{
\begin{array}{lcl}
-2\beta, &  a\neq 0\;,\\ [7pt]
-8\beta, &  a=0\;.
\end{array}
\right.
\end{align}

\end{corollary}
For the case of a $q$-linear lattice in Theorem \ref{main-Thm1} (respectively linear lattice in Theorem \ref{mainthmquadratics}) we refer the reader to the reference \cite{RKDP2020} (respectively \cite{KCJPLN, MP1994}).

\section{Applications: characterization theorems}
The objective of this paragraph is to show how previous results can be used to solve some interesting problems. For this purpose let us recall some monic OPS needed here. We will consider a monic $(P_n)_{n\geq 0}$ satisfying the following three term recurrence relation (TTRR): 
\begin{align}\label{TTRR_relation}
P_{-1} (z)=0, \quad P_{n+1} (z) =(z-B_n)P_n (z)-C_n P_{n-1} (z) \quad (C_n \neq 0)\;,
\end{align}
with $$B_n=\frac{\left\langle {\bf u}, zP_n ^2\right\rangle}{ \left\langle {\bf u}, P_n ^2\right\rangle}\;,\quad C_{n+1}=\frac{\left\langle {\bf u}, P_{n+1} ^2\right\rangle }{ \left\langle {\bf u}, P_n ^2\right\rangle}    \;.$$
Recall that the monic Askey-Wilson polynomial, $(Q_n(\cdot; a_1, a_2, a_3, a_4 | q))_{n\geq 0}$, satisfy \eqref{TTRR_relation} (see \cite[(14.1.5)]{KLS2010}) with
\begin{align*}
B_n &= a_1+\frac{1}{a_1}-\frac{(1-a_1a_2q^n)(1-a_1a_3q^n)(1-a_1a_4q^n)(1-a_1a_2a_3a_4q^{n-1})}{a_1(1-a_1a_2a_3a_4q^{2n-1})(1-a_1a_2a_3a_4q^{2n})}\\[7pt]
&\quad-\frac{a_1(1-q^n)(1-a_2a_3q^{n-1})(1-a_2a_4q^{n-1})(1-a_3a_4q^{n-1})}{(1-a_1a_2a_3a_4q^{2n-1})(1-a_1a_2a_3a_4q^{2n-2})},\\[7pt]
C_{n+1}&=(1-q^{n+1})(1-a_1a_2a_3a_4q^{n-1}) \\[7pt]
&\quad\times \frac{(1-a_1a_2q^n)(1-a_1a_3q^n)(1-a_1a_4q^n)(1-a_2a_3q^n)(1-a_2a_4q^n)(1-a_3a_4q^n)}{4(1-a_1a_2a_3a_4q^{2n-1})(1-a_1a_2a_3a_4q^{2n})^2 (1-a_1a_2a_3a_4q^{2n+1})}
\end{align*}
and subject to the following restrictions (see \cite{KCDMJP2022classical}):
$$
\begin{array}l
(1-a_1a_2a_3a_4q^n)(1-a_1a_2q^n)(1-a_1a_3q^n) \\[7pt]
\qquad\quad\times(1-a_1a_4q^n)(1-a_2a_3q^n)(1-a_2a_4q^n)(1-a_3a_4q^n) \neq 0\;.
\end{array}
$$
The monic Meixner polynomials of the second kind, $(M_n(\cdot;b_1,b_2))_{n\geq 0}$, are defined by (see \cite[p.179, (3.17)]{C1978})
\begin{align}
zM_n(z;b_1,b_2)&=M_{n+1}(z;b_1,b_2)-b_1(2n+b_2) M_n(z;b_1,b_2) \label{Meixner2nd} \\ 
&\quad +(b_1^2+1)n(n+b_2-1)M_{n-1}(z;b_1,b_2),\quad M_{-1}(z;b_1,b_2)=0\;, \nonumber
\end{align}
where $b_1$ and $b_2$ are parameters so that $b_1^2\neq-1$ and $b_2\neq0,-1,-2,\ldots$, while the monic Al-Salam polynomials $(A_n(\cdot;a,b|q))_{n\geq 0}$ satisfy \eqref{TTRR_relation} with $B_n=(a+b)q^n/2$ and $C_{n+1}=(1-abq^n)(1-q^{n+1})/4$ for each $n=0,1,\ldots$ (see \cite{I2005}). In addition, the continuous monic dual $q$-Hahn polynomials, $(R_n(x;a,b|q))_{n\geq 0}$, satisfies \eqref{TTRR_relation} with
$a_n=(a+a^{-1}-a(1-q^n)(1-bcq^{n-1}) -(1-abq^n)(1-acq^n)/a)/2$ and $b_n=(1-abq^n)(1-acq^n)(1-bcq^n)(1-q^{n+1})/4$ .

\subsection{A first case}
Let $(P_n)_{n\geq 0}$ be a monic OPS such that 
\begin{align}\label{first-charact}
\mathrm{D}_x P_{n+1}(z)=k_n\mathrm{S}_xP_{n}(z)\quad (n=0,1,\ldots;\;k_n\in \mathbb{C})\;.
\end{align}
Our goal in this section is to characterize all such OPS for quadratic and $q$-quadratic lattices. These results are presented in \cite{KCDMJP2021rama}.

\begin{lemma}\label{lemma-exple1}
Let $(P_n)_{n\geq 0}$ be a monic OPS with respect to ${\bf u}\in \mathcal{P}^*$. Assume that \eqref{first-charact} holds. Then ${\bf u}$ is $x$-classical. Moreover, 
\begin{align}
{\bf D}_x(\phi {\bf u})={\bf S}_x (\psi {\bf u})\;, \label{true-x-GP-equation}
\end{align}
where
\begin{align*}
\psi(z)=B_0-z,\quad 
\phi(z)=\left\{
\begin{array}{lcl}
(\alpha -\alpha^{-1})(z-\mathfrak{c}_3)(z-B_0)+\alpha^{-1}C_1, &  q\neq1\;,\\ [7pt]
2\beta(z-B_0)+C_1, &  q =1\;.
\end{array}
\right.
\end{align*}
\end{lemma}

\begin{proof}
By identifying the leading coefficient on each member of \eqref{first-charact} using \eqref{Dx-xn}--\eqref{Sx-xn}, we obtain $k_n=\alpha_n ^{-1}\gamma_{n+1}$ . Let $({\bf a}_n)_{n\geq 0}$ and $({\bf a} ^{[1]} _n)_{n\geq 0}$ be the dual basis associated to the sequences $(P_n)_{n\geq 0}$ and $(P ^{[1]} _n)_{n\geq 0}$, respectively. We claim that 
\begin{align}\label{kn-and-Sxan1}
{\bf S}_x{\bf a}^{[1]} _n=\alpha_n {\bf a}_n\;.
\end{align}
Indeed, by \eqref{first-charact}, we have 
$$\left\langle {\bf S}_x {\bf a}^{[1]} _n,P_j  \right\rangle =\left\langle  {\bf a}^{[1]} _n,\mathrm{S}_xP_j  \right\rangle = k_j ^{-1}\gamma_{j+1} \left\langle {\bf a}^{[1]} _n,P^{[1]} _{j}  \right\rangle=\alpha_j \delta_{n,j}, \quad (j=0,1,\dots)\;.$$ 
Hence
$$ {\bf S}_x {\bf a}_n ^{[1]} =\sum_{j=0} ^{+\infty} \left\langle {\bf S}_x {\bf a}^{[1]} _n,P_j  \right\rangle {\bf a}_j=\alpha_n {\bf a}_n\;.$$
We now apply the operator ${\bf D}_x$ to \eqref{kn-and-Sxan1}, using \eqref{def Dx^nSx-u} (for $n=1$ and replacing ${\bf u}$ by ${\bf a}_n ^{[1]}$) and  the fact that
${\bf D}_x  {\bf a}^{[1]} _n=-\gamma_{n+1}{\bf a}_{n+1}$, to obtain
\begin{align*}
-\alpha\alpha_n {\bf D}_x {\bf a}_n =(2\alpha ^2-1)\gamma_{n+1}{\bf S}_x {\bf a}_{n+1} +\gamma_{n+1}\texttt{U}_1 {\bf D}_x {\bf a}_{n+1}\;.
\end{align*}
Taking into that
$\texttt{U}_1 {\bf D}_x {\bf a}_{n+1}=\alpha {\bf D}_x (\texttt{U}_1 {\bf a}_{n+1})-(\alpha^2-1){\bf S}_x {\bf a}_{n+1}$ holds, the previous equality becomes
\begin{align*}
{\bf D}_x (\alpha_n {\bf a}_{n}+\gamma_{n+1}\texttt{U}_1{\bf a}_{n+1})=-\alpha \gamma_{n+1}{\bf S}_x {\bf a}_{n+1}\;.
\end{align*}
The desired result follows from the previous equation by taking $n=0$ and by using successively \eqref{expression-an} and \eqref{TTRR_relation}.
\end{proof}

\begin{theorem}\label{ThmMain1}
Consider the lattice $x(s)=\mathfrak{c}_1q^{-s} +\mathfrak{c}_2 q^s +\mathfrak{c}_3$ with $\mathfrak{c}_1\mathfrak{c}_2\neq 0$. 
Then, up to an affine transformation of the variable, the only monic OPS, $(P_n)_{n\geq 0}$, satisfying 
\begin{align}\label{equat-solving}
\mathrm{D}_xP_{n+1}(z)=\alpha_n^{-1}\gamma_{n+1}\mathrm{S}_xP_n(z)\;,
\end{align}
are those of the Askey-Wilson polynomials
$$
P_n(z)=2^n(\mathfrak{c}_1\mathfrak{c}_2)^{n/2}Q_n\left(\frac{z-\mathfrak{c}_3}{2\sqrt{\mathfrak{c}_1\mathfrak{c}_2}};a,-a,iq^{-1/2}/a,-iq^{-1/2}/a\Big|q\right)\;, 
$$
with $a\notin \left\lbrace \pm q^{(n-1)/2}, \pm iq^{-n/2}\,|\,n=0,1,\ldots\right\rbrace$.
\end{theorem}

\begin{proof}
Let ${\bf u}\in \mathcal{P}^*$ be the regular functional with respect to which $(P_n)_{n\geq 0}$ is an OPS.  We claim that $B_n$, in \eqref{TTRR_relation}, is given by
\begin{align}\label{Bn-sol-q-quadratic}
B_n =\mathfrak{c}_3\quad (n=0,1,\ldots)\;.
\end{align}  
Indeed, recall that $P_n(z)=z^n+f_n z^{n-1}+\cdots$, where $f_0=0$ and $f_n=-\sum_{j=0} ^{n-1}B_j$  $(n=1,2,\dots)$.
We then use \eqref{Dx-xn}--\eqref{Sx-xn} and identify the second coefficient of higher degree in both sides of \eqref{equat-solving}. This gives
$$
\frac{\alpha_n}{\gamma_{n+1}}f_{n+1}=\frac{\alpha_{n-1}}{\gamma_n}f_n +\frac{n\alpha-\alpha_n\gamma_n}{\gamma_{n}\gamma_{n+1}}\mathfrak{c}_3\;.
$$
By the telescopic sum method, we have
\begin{align*}
f_n=-\frac{\gamma_n}{\alpha_{n-1}}\left(B_0 +\mathfrak{c}_3\sum_{j=1} ^{n-1}\frac{\alpha_j\gamma_j- \alpha j}{\gamma_{j}\gamma_{j+1}}    \right)=-\frac{\gamma_n}{\alpha_{n-1}}(B_0-\mathfrak{c}_3)-n\mathfrak{c}_3\;.
\end{align*}
Therefore we obtain
\begin{align}\label{first-Bn-q}
B_n =f_n-f_{n+1}=\mathfrak{c}_3+\frac{\alpha}{\alpha_{n-1}\alpha_n}(B_0-\mathfrak{c}_3)\;.
\end{align}
By Lemma \ref{lemma-exple1}, ${\bf u}$ satisfies \eqref{true-x-GP-equation} where $\phi(z)=-(\alpha-\alpha^{-1})(z-B_0)(z-\mathfrak{c}_3)-\alpha^{-1}C_1$ and $\psi(z)=z-B_0$. From \eqref{Bn-Dx} we obtain
\begin{align}\label{second-Bn-q}
B_n=\mathfrak{c}_3 +(1+q)(B_0-\mathfrak{c}_3)q^{n-2}\frac{(q-1)(1-q^{2n-2})+(1+q)q^{n-1}}{(1+q^{2n-3})(1+q^{2n-1})}\;.
\end{align} 
It is not hard to see that the equality between \eqref{first-Bn-q} and \eqref{second-Bn-q} requires $B_0=\mathfrak{c}_3$ and so \eqref{Bn-sol-q-quadratic} follows.
Therefore, the above expressions for $\phi$ and $\psi$ reduce to
$$\phi(z)=-(\alpha-\alpha^{-1})(z-\mathfrak{c}_3)^2-\alpha^{-1}C_1\;,\quad
\psi(z)=z-\mathfrak{c}_3\;.$$ 
It follows that $C_1$ can be regarded as the only free parameter. 
From \eqref{Cn-Dx} we obtain $d_n= \alpha^{-1}\alpha_{n-1}$ and
\begin{align}
\phi ^{[n]}\Big(\mathfrak{c}_3-\frac{e_n}{d_{2n}}\Big)&=\frac{1}{2\alpha}\Big(\mathfrak{c}_1\mathfrak{c}_2(1-q^{-1})(1-q^{n})(1+q^{-n+1})-2C_1  \Big)\;. \label{Eqphinc3}
\end{align}
Choose a parameter $r$ as a solution of the quadratic equation
$$
(q-1)\mathfrak{c}_1\mathfrak{c}_2 Z^2 +2(C_1 +2(\alpha^2-1)\mathfrak{c}_1\mathfrak{c}_2)Z-(1-q^{-1})\mathfrak{c}_1\mathfrak{c}_2=0\;;
$$
that is
$$
r= \frac{C_1+2(\alpha^2-1)\mathfrak{c}_1\mathfrak{c}_2}{(1-q)\mathfrak{c}_1\mathfrak{c}_2} \pm \sqrt{q^{-1}+\Big( \frac{C_1+2(\alpha^2-1)\mathfrak{c}_1\mathfrak{c}_2}{(1-q)\mathfrak{c}_1\mathfrak{c}_2} \Big)^2}\;.
$$ 
Instead of $C_1$ we may consider $r$ as the free parameter and to express $C_1$ in terms of $r$ as follows:
$$C_1=\frac12(1-q^{-1})(1+r^{-1})(1-rq)\mathfrak{c}_1\mathfrak{c}_2\;.$$ 
Therefore \eqref{Eqphinc3} can be rewriten as
\begin{align*}
\phi ^{[n]}\Big(\mathfrak{c}_3-\frac{e_n}{d_{2n}}\Big)&=\mathfrak{c}_1\mathfrak{c}_2\frac{1-q}{2\alpha}(1+rq^{n})(1-r^{-1}q^{n-1})q^{-n}. 
\end{align*}
It follows from the regularity conditions that the free parameter $r$ should satisfy the condition $r\notin \left\lbrace q^{n-1}, -q^{-n}\,|\,n=0,1,\ldots\right\rbrace$.
Moreover, \eqref{Cn-Dx} yields
\begin{align}
C_{n+1}=\mathfrak{c}_1\mathfrak{c}_2 \frac{(1+q^{n-2})(1-q^{n+1})(1+rq^n)(1-r^{-1}q^{n-1})}{(1+q^{2n-2})(1+q^{2n})}.
\end{align}
Thus 
$$P_n(x)=2^n(\mathfrak{c}_1\mathfrak{c}_2)^{n/2}Q_n\left(\frac{z-\mathfrak{c}_3}{2\sqrt{\mathfrak{c}_1\mathfrak{c}_2}};\sqrt{r},-\sqrt{r},i/\sqrt{rq},-i/\sqrt{rq} \,\Big|\,q\right),$$ 
and setting $r=a^2$ the theorem follows.
\end{proof}

Following the same ideas, we may easily prove rigorously the following one.
 
\begin{corollary}\label{main-thm-quadratic}
Consider the lattice  $x(s)=4\beta s^2 +\mathfrak{c}_5 s +\mathfrak{c}_6$ with $(\beta,\mathfrak{c}_5)\neq (0,0)$.
Then there exist monic OPS, $(P_n)_{n\geq 0}$, satisfying 
\begin{align}\label{equat-solving2}
\mathrm{D}_xP_{n+1}(z)=(n+1)\mathrm{S}_xP_n(z)
\end{align}
if and only if $\beta=0$. In this case, up to an affine transformation of the variable, these polynomials are those of Meixner of the second kind
$$
P_n(z)=\left(\frac{i\mathfrak{c}_5}{2}\right)^nM_n\left(\frac{2i(B_0-z)}{\mathfrak{c}_5};0,-\frac{4C_1}{\mathfrak{c}_5 ^2}\right),
$$
with $B_0,C_1\in\mathbb{C}$ and  $4C_1/\mathfrak{c}_5^2\not\in\mathbb{N}_0$.
\end{corollary}

\subsection{A second case}

In this section instead of \eqref{first-charact}, we characterize monic OPS satisfying the following structure relation
\begin{align}\label{second_charact}
\mathrm{D}_xP_{n}(z)=c_n P_{n-1}(z) \quad (n=0,1,\ldots\;.)
\end{align} 
for quadratic and $q$-quadratic lattices. We start by showing that all monic OPS satisfying \eqref{second_charact} are $x$-classical and then we prove that the coefficients of the associated TTRR satisfy a system of non linear equations. These results appear in \cite{DMthesis}.
\begin{lemma}\label{conject-is-dx-classical}
We consider the $q$-quadratic lattice $x(s)=\mathfrak{c}_1q^{-s}+\mathfrak{c}_2q^s+\mathfrak{c}_3$. Let ${\bf u} \in \mathcal{P}^* $ be a regular functional such that its corresponding monic OPS $(P_n)_{n\geq0}$ satisfies \eqref{second_charact} subject to the condition $c_n\neq 0$, for each $n=0,1,\ldots$. Then ${\bf u}$ is $x$-classical. Moreover $D_x(\phi {\bf u})=S_x(\psi {\bf u})$, with $\psi$ and $\phi$ polynomials given by
\begin{align}\label{value-d-e-a}
\psi(z)=z-B_0,\quad \phi(z)=(\mathfrak{a}z-\mathfrak{b})(z-B_0 )-(\mathfrak{a}+\alpha)C_1\; ,
\end{align}
where
\begin{align}\label{coef-a-b-conject-classical}
\mathfrak{a}:= \alpha \frac{2C_1-C_2}{C_2} \;,
\quad \mathfrak{b}:= \beta -B_0 +2\alpha \frac{B_1C_1}{C_2} \;.
\end{align}
(Here, $B_0$, $B_1$, $C_1$, and $C_{2}$ are coefficients of the TTRR \eqref{TTRR_relation} satisfied by $(P_n)_{n\geq0}$.)
\end{lemma}

\begin{proof}
By identifying the leading coefficients of each member of \eqref{second_charact}, we obtain $c_n=\gamma_n$. Let $({\bf a}_n)_{n\geq0}$ be the dual basis associated to the monic OPS $(P_n)_{n\geq0}$ satisfying \eqref{second_charact}. We claim that 
\begin{align}\label{def-R1}
{\bf D}_x {\bf u} = R_1 {\bf u},\quad R_1(z) := \frac{B_0-z}{C_1}\;.
\end{align}
Indeed fix $j \in \mathbb{N}_0$. Using \eqref{second_charact}, we deduce
$\left\langle {\bf D}_x  {\bf a}_0,P_j \right\rangle =-\left\langle {\bf a}_0, D_x P_j \right\rangle =-\gamma_j \delta_{0,j-1}\;.$ This implies
$$
{\bf D}_x  {\bf a}_0 =\sum_{j=0} ^{\infty} \left\langle {\bf D}_x  {\bf a}_0,P_j \right\rangle {\bf a}_j =-{\bf a}_1\;.
$$ 
Hence \eqref{def-R1} holds using \eqref{expression-an} and \eqref{TTRR_relation}. Applying $D_x$ to both sides of (\ref{TTRR_relation}), and using (\ref{1.3}), yields
\begin{align*}
S_x P_n(z) =-(\alpha z+\beta)D_x P_n(z) +
D_x P_{n+1}(z) +B_n D_x P_n(z) +C_n D_x P_{n-1}(z)\;.
\end{align*}
Using \eqref{second_charact} and (\ref{TTRR_relation}), we obtain
\begin{align}\label{piSx-eq}
S_x P_n (z)=r_n ^{[1]} P_{n}(z)+r_n ^{[2]} P_{n-1}(z)+r_n ^{[3]}P_{n-2} (z) 
\end{align}
for each $n=0,1,2,\ldots$, where
\begin{align*}
r_n ^{[1]}:= c_{n+1}-\alpha c_n\;,\quad r_n ^{[2]}:= c_n(B_n-\alpha B_{n-1}-\beta)\;,~
r_n ^{[3]}:= C_nc_{n-1}-\alpha C_{n-1}c_n\;.
\end{align*}
For a fixed $j\in \mathbb{N}_0$, using \eqref{piSx-eq} we obtain 
$$
\left\langle {\bf S}_x {\bf a}_0,P_j  \right\rangle = \left\langle {\bf a}_0, S_x P_j\right\rangle 
= r_j ^{[1]}\delta_{0,j} +r_j ^{[2]}\delta_{0,j-1} +r_j ^{[3]}\delta_{0,j-2}\;.
$$
Therefore, 
$$  {\bf S}_x  {\bf a}_0 =\sum_{j=0} ^{\infty} \left\langle {\bf S}_x  {\bf a}_0,P_j  \right\rangle {\bf a}_j 
= r_0 ^{[1]} {\bf a}_0 +r_1 ^{[2]} {\bf a}_1 +r_2 ^{[3]} {\bf a}_2\;, $$ 
and so
\begin{align}\label{def-R2}
{\bf S}_x {\bf u} = R_2 {\bf u},\quad R_2(z) := r_0 ^{[1]} +\frac{r_1 ^{[2]}}{C_1} P_1 (z)+\frac{r_2 ^{[3]}}{C_2 C_1}P_2 (z)\;.
\end{align}
Next, on the first hand, applying successively \eqref{def-R2}, \eqref{def Dx^nSx-u} and \eqref{def-R1}, we obtain
\begin{align}\label{first-hand}
{\bf D}_x (R_2 {\bf u}) &= {\bf D}_x {\bf S}_x  {\bf u}=\frac{2\alpha ^2 -1}{\alpha} {\bf S}_x (R_1 {\bf u}) +\frac{\texttt{U}_1}{\alpha} {\bf D}_x  (R_1 {\bf u})\;.
\end{align}
On the other hand, using \eqref{a} with $f=\texttt{U}_1$, we obtain 
\begin{align}\label{second-hand}
\frac{\texttt{U}_1}{\alpha} {\bf D}_x (R_1 {\bf u}) ={\bf D}_x (\texttt{U}_1 R_1 {\bf u}) -\frac{\alpha^2 -1}{\alpha}{\bf S}_x (R_1 {\bf u})\;.
\end{align}
Thus, combining \eqref{first-hand} and \eqref{second-hand}, we obtain 
$$
{\bf D}_x \Big((R_2 -\texttt{U}_1 R_1){\bf u}\Big)={\bf S}_x (\alpha R_1 {\bf u})\;.
$$
This leads us to define
$
\psi(x):=z-B_0$ and $\phi(z):=-\alpha ^{-1} C_1\big( R_2 (z)-\texttt{U}_1 (z)R_1 (z)\big)$. Clearly, $\deg \psi =1$, $\deg \phi \leq 2$ and  ${\bf D}_x(\phi {\bf u})={\bf S}_x (\psi {\bf u})$.
Finally, setting without lost of generality $c_0:=0$ and $C_0:=0$, we obtain
\begin{align*}
\mathfrak{a}:&=\frac{1}{2}\phi''(0)= -\frac{1}{\alpha } \left( \frac{r_2 ^{[5]}}{ C_2} +\alpha ^2 -1 \right)
=2\alpha \frac{C_1}{C_2}-\alpha\;.
\end{align*}
Similarly, we also have
\begin{align*}
&\phi'(0)= -\mathfrak{a}B_0  -\beta +B_0 -(\mathfrak{a}+\alpha)B_1 \;,  \\
&\phi(0)=-(\mathfrak{a}+\alpha)C_1 -B_0 \left( -\beta +B_0 -(\mathfrak{a}+\alpha)B_1  \right)\; .
\end{align*} 
Hence the desired result is proved. 
\end{proof}

\begin{lemma}\label{system-to-solve}
For the $q$-quadratic lattice $x(s)=\mathfrak{c}_1q^{-s}+\mathfrak{c}_2q^s+\mathfrak{c}_3$, 
let $(P_n)_{n\geq0}$ be a monic OPS satisfying \eqref{second_charact}. Then the coefficients $B_n$ and $C_{n}$ of the TTRR \eqref{TTRR_relation} satisfied by $(P_n)_{n\geq0}$ fulfill the following system of difference equations: 
\begin{align}
&\; c_{n+2}-2\alpha c_{n+1}+c_n=0 \;, \label{eq1S} \\
&\; t_{n+2} -2\alpha t_{n+1}+t_{n} =0 \;,\quad t_n :=c_n/C_n =k_1 q^{n/2}+k_2q^{-n/2}\;, \label{eq2S} \\
&t_{n+3}(B_{n+2}-\mathfrak{c}_3) -(t_{n+2} +t_{n+1})(B_{n+1}-\mathfrak{c}_3)  +t_n (B_n -\mathfrak{c}_3)=0\;,  \label{eq3S} \\
\nonumber \\
& \begin{array}{l} \label{eq4S}
(t_{n+1}+t_{n+2})(C_{n+1}-\mathfrak{c}_1 \mathfrak{c}_2) -2(1+\alpha)t_n(C_n -\mathfrak{c}_1 \mathfrak{c}_2) +(t_{n-1}+t_{n-2})(C_{n-1} -\mathfrak{c}_1 \mathfrak{c}_2) \\ [0.2em]
\quad = t_n \left[(B_{n}-\mathfrak{c}_3)^2 -2\alpha (B_{n}-\mathfrak{c}_3)(B_{n-1}-\mathfrak{c}_3) +(B_{n-1}-\mathfrak{c}_3)^2   \right] \;,
\end{array}  \\
\nonumber \\
& \begin{array}{l}  
c_{n+1}(B_{n+1}-\mathfrak{c}_3)+(1-2\alpha)(B_n-\mathfrak{c}_3)+c_n(B_{n-1}-\mathfrak{c}_3)=0 \;.\label{eq5S}
\end{array} 
\end{align}
\end{lemma}

\begin{proof}
Applying the operator $S_x$ to both sides of (\ref{TTRR_relation}) and using (\ref{1.4}), we deduce
\begin{align*}
\texttt{U}_2(z) D_x P_n(z) + (\alpha z+\beta)S_x P_n(z) =S_x P_{n+1}(z) +B_n S_x P_n(z) +C_n S_x P_{n-1}(z)\,.
\end{align*}
Using successively (\ref{U2-simples-bis}), (\ref{second_charact}), (\ref{piSx-eq}), and (\ref{TTRR_relation}), we obtain
a vanishing linear combination of the polynomials $P_{n+3}$, $P_{n+2}$,..., $P_{n-3}$.
Thus, setting
\begin{equation}\label{def-tn}
t_n:=c_n/C_n,\quad n=1,2,3,\ldots\;, 
\end{equation}
after straightforward computations we obtain \eqref{eq1S}--\eqref{eq2S} together with the following equations:
\begin{align}
& \begin{array}{l}
-t_{n+2}B_{n+1} +(t_{n+1}+t_{n})B_n-t_{n-1}B_{n-1}=2\beta(t_n+t_{n+1}) \;,
\end{array} \label{seq4S} \\
& \begin{array}{l}
t_nB_{n}^2 + t_n B_{n-1}^2 -(t_{n-1}+t_{n+1})B_n B_{n-1}-2\beta t_n(B_n+B_{n-1})  \\ [0.2em]
\quad -(t_{n+2}+t_{n+1})C_{n+1}+ 2(1+\alpha)t_nC_n -(t_{n-2}+t_{n-1})C_{n-1}=(\delta-\beta^2)t_n \;,\\ [0.2em]
\end{array} \label{seq6S} \\
& \begin{array}{l}
c_{n+1}B_{n+1}+(1-2\alpha)(c_n+c_{n+1})B_n +c_nB_{n-1}=2\beta(c_n+c_{n+1}) \,,
\end{array} \label{seq7S}  
\end{align}
where $\delta=(\alpha^2-1)(\mathfrak{c}_3 ^2-4\mathfrak{c}_2\mathfrak{c}_2)$.
\eqref{eq3S} (respectively \eqref{eq4S} and \eqref{eq5S}) is obtained from \eqref{seq4S} (respectively \eqref{seq6S} and \eqref{seq7S}). This completes the proof.  
\end{proof}

\begin{remark}
According to Lemma \ref{system-to-solve}, the coefficients $B_n$ and $C_n$ of the TTRR \eqref{TTRR_relation} of any monic OPS $(P_n)_{n\geq 0}$ fulfilling \eqref{second_charact} must fulfill \eqref{eq1S}--\eqref{eq5S}. However, we need to take into account some initial conditions which will be specified. Indeed, for instance, it is clear that
$$B_n=\mathfrak{c}_3,\quad \quad C_{n+1}=\mathfrak{c}_1\mathfrak{c}_2\quad (n=0,1,2,\ldots)\;,$$ provide a solution of the system \eqref{eq1S}--\eqref{eq5S}. The corresponding monic OPS is
$$P_n (x)=2^n (\mathfrak{c}_1\mathfrak{c}_2)^{n/2}\widehat{ U}_n \left(\frac{z-\mathfrak{c}_3}{2\sqrt{\mathfrak{c}_1\mathfrak{c}_2}}\right)~~~~(n=0,1,2,\ldots)\;, $$ where $(\widehat{U}_n)_{n \geq 0}$ is the monic Chebyschev polynomials of the second kind. However this sequence $(P_n)_{n \geq 0}$ does not provide a solution of \eqref{second_charact} (see \eqref{C2-true-pi-0} below). 
\end{remark}

 The system of equations \eqref{eq1S}--\eqref{eq5S} is non-linear and so, in general it is not easy to solve it. Nevertheless, in view of Lemma \ref{conject-is-dx-classical}, an OPS satisfying \eqref{second_charact} is $x$-classical and so the results presented in previous paragraph will be useful to find the explicit expressions for the coefficients of the TTRR satisfied by the OPS under analysis (see Theorem \ref{main-Thm1} and Corollary \ref{assymptotic-behaviour-coef-ttrr-Dx}). We will see that some patterns appear associated with the system of equations \eqref{eq1S}--\eqref{eq5S} which will allow us to solve the system. 

Recall that from \eqref{eq2S}, we have 
\begin{align}\label{general-tn-struc}
t_n = \frac{c_n}{C_n} =k_1 q^{n/2}+k_2q^{-n/2}\quad (n=1,2,3,\ldots)\;,
\end{align} 
where $k_1$ and $k_2$ are two complex numbers. Since $c_n \neq 0$, for $n=1,2,3,\ldots$, then $k_1$ and $k_2$ cannot vanish simultaneously. Recall also that we defined $c_0=C_0=0$, and so we define $$t_0:=k_1+k_2\;,$$ by compatibility with \eqref{general-tn-struc}.

\begin{theorem}
Up to an affine transformation of the variable, the only monic OPS $(P_n)_{n\geq 0}$ satisfying \eqref{second_charact} are the  monic $q$-Hermite polynomials of Rogers.
\end{theorem}

\begin{proof}
Let $(P_n)_{n\geq 0}$ be a monic OPS satisfying \eqref{second_charact}. We know that $c_n=\gamma_n$ and therefore by identifying the coefficient of the second term with the higher degree in \eqref{second_charact}, we also obtain $$B_n=\mathfrak{c}_3+(\gamma_{n+1}-\gamma_n)(B_0-\mathfrak{c}_3)\;. $$ 
Since by assumption $t_n\neq 0$ for each $n=1,2,\ldots$, it is not hard to see that the above expression of $B_n$ satisfies \eqref{eq3S} if and only if $B_0=\mathfrak{c}_3$.
Hence
\begin{align}\label{Bn-power-pi-0}
B_n =\mathfrak{c}_3 \quad (n=0,1,2,\ldots)\;.
\end{align} 
In addition, from \eqref{Bn-Dx} we obtain 
$$
\frac{\gamma_{n+1}e_n}{d_{2n}}=\frac{\gamma_n e_{n-1}}{d_{2n-2}}\quad (n=0,1,2,\ldots)\;.
$$
Since $e_n =\phi'(\mathfrak{c}_3)\gamma_n +\psi(\mathfrak{c}_3)\alpha_n$, we find $e_0=0$ (because $\gamma_0=0$, and from \eqref{value-d-e-a}, $\psi(\mathfrak{c}_3)=\mathfrak{c}_3-B_0=0$) and so $e_n=0$, for $n=0,1,2,\ldots$, and consequently $\phi'(\mathfrak{c}_3)=0$. Then, from \eqref{value-d-e-a} we obtain $\mathfrak{b}=\mathfrak{a}\mathfrak{c}_3$.
Taking $n=3$ in \eqref{second_charact} and using \eqref{Dx-xn}--\eqref{vnD}, we obtain
\begin{align}\label{C2-true-pi-0}
C_2=2(2\alpha^2-1)(C_1-\mathfrak{c}_1\mathfrak{c}_2)+2\mathfrak{c}_1\mathfrak{c}_2\;.
\end{align}
Therefore, using \eqref{value-d-e-a}, we have $\phi(z)=\mathfrak{a}(z-\mathfrak{c}_3)^2 -(\mathfrak{a}+\alpha)C_1$ and $\psi(z)=z-\mathfrak{c}_3$, where
\begin{align}\label{value-a-pi-0}
\mathfrak{a}=\frac{2\alpha C_1}{C_2}-\alpha &=\frac{2\alpha(1-\alpha^2)(C_1-\mathfrak{c}_1\mathfrak{c}_2)}{(2\alpha^2 -1)(C_1-\mathfrak{c}_1\mathfrak{c}_2)+\mathfrak{c}_1\mathfrak{c}_2}\;.
\end{align}

Recall that
\begin{align}
t_n =\frac{\gamma_n}{C_n}=k_1q^{n/2}+k_2q^{-n/2},\quad k_1 = \frac{(1+q)C_1-C_2}{q^{1/2}(q-1)C_1C_2},\quad k_2= \frac{(q^{-1}+1)C_1-C_2}{q^{-1/2}(q^{-1}-1)C_1C_2} \;.\label{tn-pi-0}
\end{align}
Equation \eqref{eq4S} reduces to 
\begin{align}\label{Cn-pi-0}
(t_{n+1}+t_{n+2})(C_{n+1}-\mathfrak{c}_1 \mathfrak{c}_2) -2(1+\alpha)t_n(C_n -\mathfrak{c}_1 \mathfrak{c}_2) +(t_{n-1}+t_{n-2})(C_{n-1} -\mathfrak{c}_1 \mathfrak{c}_2) =0\;.
\end{align} 
Next, define $r_n:=t_n +t_{n+1} =aq^{n/2} +bq^{-n/2}$, where $a:=k_1(1+q^{1/2})$ and $b:=k_2(1+q^{-1/2})$. By setting $K_c:=\Big(r_2 (C_2-\mathfrak{c}_1\mathfrak{c}_2)-r_0 (C_1-\mathfrak{c}_1\mathfrak{c}_2)\Big)/(1-q^{-1/2})$, we see that \eqref{Cn-pi-0} reads as
\begin{align*}
r_{n+1}(C_{n+1}-\mathfrak{c}_1\mathfrak{c}_2 ) - r_{n-1}(C_n-\mathfrak{c}_1\mathfrak{c}_2)=r_{n}(C_n-\mathfrak{c}_1\mathfrak{c}_2) -r_{n-2}(C_{n-1}-\mathfrak{c}_1\mathfrak{c}_2)\;.
\end{align*}
By applying this relation successively, we obtain
\begin{align*}
r_{n+1}(C_{n+1}-\mathfrak{c}_1\mathfrak{c}_2 ) - r_{n-1}(C_n-\mathfrak{c}_1\mathfrak{c}_2)=K_c(1-q^{-1/2})\;.
\end{align*}
Multiplying this equation by $r_n$ and applying  a telescoping process to the resulting equation, we obtain
\begin{align}\label{Cn1-pi-0}
C_{n+1}= \mathfrak{c}_1\mathfrak{c}_2 +  \frac{r_0r_1(C_1-\mathfrak{c}_1\mathfrak{c}_2)q^{n/2} +K_c\Big(aq^n +(b-aq^{-1/2})q^{n/2}-bq^{-1/2}\Big)}{(aq^{n} +b)(aq^{n+1} +b)}q^{(n+1)/2} 
\end{align}
for each $n=0,1,2,\ldots$.\\
We claim that 
\begin{align}\label{k1k2-is-0-pi-0}
\Big(C_2 -(1+q)C_1\Big)\Big(C_2-(1+q^{-1})C_1\Big)=0\;.
\end{align}
Indeed, suppose that \eqref{k1k2-is-0-pi-0} does not hold. Then, by \eqref{tn-pi-0}, we would have $k_1k_2\neq 0$, and we may write 
\begin{align}\label{Cn1a-pi-0}
C_{n+1}=\frac{\gamma_{n+1}}{t_{n+1}}=\frac{u(q^{n+1}-1)}{k_1q^{n+1}+k_2},~~~u^{-1}=q^{1/2}-q^{-1/2}~~(n=0,1,2,\ldots)\;.
\end{align}
Assume without lost of generality that $0<q<1$. Then taking successively limits as $n\to\infty$ in the expressions  for $C_{n+1}$ and $q^{-(n+1)/2}(C_{n+1}-\mathfrak{c}_1\mathfrak{c}_2)$ given by \eqref{Cn1-pi-0} and \eqref{Cn1a-pi-0}, we obtain $u+k_2\mathfrak{c}_1\mathfrak{c}_2=0$ and $K_c=0$. Now, the equality between \eqref{Cn1-pi-0} and \eqref{Cn1a-pi-0} implies 
\begin{align*}
2(1+\alpha)(u-k_1\mathfrak{c}_1\mathfrak{c}_2)k_1 ^2 q^{2n+2}+k_1\Big(4(1+\alpha)^2(u-k_1\mathfrak{c}_1\mathfrak{c}_2)k_2-r_0r_1(C_1-\mathfrak{c}_1\mathfrak{c}_2) \Big) q^{n+1}  \\
+k_2\Big(2(1+\alpha)(u-k_1\mathfrak{c}_1\mathfrak{c}_2)k_2-r_0r_1(C_1-\mathfrak{c}_1\mathfrak{c}_2) \Big) =0\;, 
\end{align*} 
for each $n=0,1,2,\ldots$. This implies $u-k_1\mathfrak{c}_1\mathfrak{c}_2=0$ and $r_0r_1(C_1-\mathfrak{c}_1\mathfrak{c}_2)=0$. Consequently, $C_{n+1}=\mathfrak{c}_1\mathfrak{c}_2$ for each $n=0,1,2,\ldots$, which contradicts \eqref{C2-true-pi-0}. Hence \eqref{k1k2-is-0-pi-0} holds and so $k_1k_2=0$.   

 Suppose that $k_1=0$, i.e., $C_2=(1+q)C_1$. Then, from \eqref{C2-true-pi-0}, we find $C_1=(1-q)\mathfrak{c}_1\mathfrak{c}_2$ and so, from \eqref{value-a-pi-0}, we obtain $\mathfrak{a}=-1/(2u)$. Since $\mathfrak{b}=\mathfrak{a}\mathfrak{c}_3$ and $B_0=\mathfrak{c}_3$, using \eqref{Cn-Dx} in Theorem \ref{main-Thm1} we obtain 
\begin{align}\label{cn-for-q-pi-0}
C_{n+1}=(1-q^{n+1})\mathfrak{c}_1\mathfrak{c}_2\quad (n=0,1,2,\ldots)\;.
\end{align}
One easily see that the expressions for $B_n$ and $C_{n+1}$ given by \eqref{Bn-power-pi-0} and \eqref{cn-for-q-pi-0} satisfy \eqref{eq5S} and so the system of equations \eqref{eq1S}--\eqref{eq5S} is satisfied.

 Similarly, if $k_2=0$, i.e. $C_2=(1+q^{-1})\mathfrak{c}_1\mathfrak{c}_2$, we obtain $\mathfrak{a}=1/(2u)$ and
\begin{align}\label{cn-for-q-1-pi-0}
C_{n+1}=(1-q^{-n-1})\mathfrak{c}_1\mathfrak{c}_2 \quad (n=0,1,2,\ldots)\;, 
\end{align}
which together with \eqref{Bn-power-pi-0} also fulfill the system of equations \eqref{eq1S}--\eqref{eq5S}.
Thus 
$$
\phi(z)=\pm \frac{1}{2u}\left( (z-\mathfrak{c}_3)^2 -\mathfrak{c}_1\mathfrak{c}_2 \right) \quad \textit{and}\quad \psi(z)=z-\mathfrak{c}_3\;.
$$ 
Hence, we conclude that
\begin{align*}
P_n (z)&=2^n (\mathfrak{c}_1\mathfrak{c}_2)^{n/2}A_n \left(\frac{z-\mathfrak{c}_3}{2\sqrt{\mathfrak{c}_1\mathfrak{c}_2}};0,0\Big|q\right) 
\end{align*}
or
\begin{align*}
P_n(z)= 2^n (\mathfrak{c}_1\mathfrak{c}_2)^{n/2}A_n\left( \frac{z-\mathfrak{c}_3}{2\sqrt{\mathfrak{c}_1\mathfrak{c}_2}};0,0\Big|q^{-1}\right),
\end{align*}
where $(A_n(;a,b|q))_{n\geq 0}$ is the monic Al-Salam polynomial. As a matter of fact, this $(P_n)_{n\geq 0}$ is the sequence of monic $q$-Hermite polynomials of Rogers. 
Thus the proof is complete.
\end{proof}

Following the ideas of this proof, one can easily prove the following result.
\begin{corollary}
Consider the quadratic lattice $x(s)=4\beta s^2 +\mathfrak{c}_5s+\mathfrak{c}_6$ with $(\beta,\mathfrak{c}_5)\neq (0,0)$. Then \eqref{second_charact} has no OPS solutions.
\end{corollary}
We denote by $\mathcal{D}_q$ the Askey-Wilson operator, obtained from the $D_x$ operator by restricting the $q$-quadratic lattice to $x(s)=(q^{-s}+q^s)/2$. It is important to notice that the previous theorem was solved in \cite{Al-Salam1995} using a different method and the advantage of this method developed here using Theorem $\ref{main-Thm1}$ lies on the fact that we can use it to solve equations of type \eqref{second_charact} in a more general form. Indeed, the following problem \cite[Conjecture 24.7.8]{I2005} is a conjecture posed by M. E. H. Ismail. 
\begin{conjecture}\label{IsmailConj2478}
Let $(P_n)_{n\geq0}$ be a monic OPS and $\pi$ be a polynomial of degree at most $2$
which does not depend on $n$. If $(P_n)_{n\geq0}$ satisfies
\begin{align}
\pi(x)\mathcal{D}_q P_n (x)= (a_n x+b_n)P_n (x) + c_n P_{n-1} (x)\;, \label{0.2Dqa}
\end{align}
then $(P_n)_{n\geq0}$ are continuous $q$-Jacobi polynomials, Al-Salam-Chihara polynomials,
or special or limiting cases of them. The same conclusion holds if $\pi$ has degree
$s+1$ and the condition $\eqref{0.2Dqa}$ is replaced by
\begin{align}\label{0.2Dq-general}
\pi(x)\mathcal{D}_qP_n(x)=\sum_{k=-r} ^{s} c_{n,k}P_{n+k}(x)\;,
\end{align}
for positive integers $r$, $s$, and a polynomial $\pi$ which does not depend on n.
\end{conjecture}
Although \eqref{0.2Dqa} is a simple relation, this problem was only solved recently due to the complexity of the Askey-Wilson operator and his properties. Using the method developed here, it is proved in \cite{KCDMJP2022conject} that the only solutions of \eqref{0.2Dqa} are some particular cases of the Al-Salam Chihara polynomials, the Chebyshev polynomials of the first kind and the continuous $q$-Jacobi polynomials. The following relation is proved in \cite{KCDM2022counterexample, DM2022firstart}:

\begin{align*}
(\alpha^2-1)(x^2-1)&\mathcal{D}_qR_{n}(x;1,-1,q^{1/4}|q^{1/2})=(\alpha^2-1)\gamma_nR_{n+1}(x;1,-1,q^{1/4}|q^{1/2})\\
&+\big(c_{n+1}-\alpha c_n +(1-\alpha)\alpha_n B_n\big)R_n(x;1,-1,q^{1/4}|q^{1/2})\\
&+\big((B_n-\alpha B_{n-1})c_n +(1-\alpha^2)\gamma_n C_n \big)R_{n-1}(x;1,-1,q^{1/4}|q^{1/2})\\
&+(c_{n-1}C_n-\alpha c_nC_{n-1})R_{n-2}(x;1,-1,q^{1/4}|q^{1/2}) \;,
\end{align*}
where $B_n$, $c_n$ and $C_n$ are given by 
\begin{align*}
B_n&=\frac{1}{2}\Big((1+q^{-1/2})q^{n/2} +1-q^{-1/2}\Big)q^{(2n+1)/4}\;,\\
C_n&=\frac{1}{4}(1+q^{(n-1)/2})(1-q^{n/2})(1-q^{n-1/2})\;,~c_n=C_nq^{-(2n-1)/4}\;,
\end{align*}
and $(R_n(\cdot;a,b,c|q))_{n\geq 0}$ the continuous monic dual $q$-Hahn polynomial.
This equation is of type \eqref{0.2Dq-general} with $\deg \pi =2$, $r=1$ and $s=2$ giving therefore a counterexample to \eqref{0.2Dq-general}.

In view of this section, it is clear that Theorem \ref{main-Thm1} presented here is useful to derive some characterization theorems on OPS and to solve some challenging problem in the subject.

\section*{Acknowledgements }
 This work is partially supported by the Centre for Mathematics of the University of Coimbra-UIDB/00324/2020, funded by the Portuguese Government through FCT/MCTES. DM is partially supported by ERDF and Consejer\'ia de Econom\'ia, Conocimiento, Empresas y Universidad de la Junta de Andaluc\'ia (grant UAL18-FQM-B025-A) and by the Research Group FQM-0229 (belonging to Campus of International Excellence CEIMAR).

\end{document}